\documentclass[12pt]{article}

\usepackage[english]{babel}
\usepackage[letterpaper,top=2cm,bottom=2cm,left=3cm,right=3cm,marginparwidth=1.75cm]{geometry}

\usepackage{amsmath}
\usepackage{amsthm}
\usepackage{amssymb}
\usepackage{xcolor}
\usepackage{hyperref}
\usepackage{tikz}
\usepackage{pgfplots}
\pgfplotsset{compat=1.18}

\theoremstyle{plain}
\newtheorem{thm}{Theorem}

\newtheorem{lemma}[thm]{Lemma}

\newtheorem{prop}[thm]{Proposition}
\newtheorem{conj}[thm]{Conjecture}

\theoremstyle{definition}
\newtheorem{defn}[thm]{Definition}

\theoremstyle{remark}
\newtheorem{rem}[thm]{Remark}
\newtheorem{ex}[thm]{Example}
\newtheorem{prob}[thm]{Problem}

\newcommand{\cF}{\mathcal{F}}

\newcommand{\bR}{\mathbb{R}}
\newcommand{\bN}{\mathbb{N}}
\newcommand{\bZ}{\mathbb{Z}}

\newcommand{\norm}[1]{\Vert#1\Vert}
\newcommand{\abs}[1]{\vert#1\vert}

\newcommand{\rar}{\rightarrow}

\title{Short spherical $t$-design curves}
\author{Emily J.\ King\thanks{Department of Mathematics, Colorado State University, Fort Collins, CO} \thanks{Corresponding author: emily.king@colostate.edu}\and Dustin G.\ Mixon\thanks{Department of Mathematics, The Ohio State University, Columbus, OH} \thanks{Translational Data Analytics Institute, The Ohio State University, Columbus, OH}}
\date{}

\begin{document}
\maketitle

\begin{abstract}
We study the minimum arclength of spherical $t$-design curves, i.e., closed rectifiable curves on $S^d$ whose normalized arclength measure exactly integrates every polynomial of degree at most $t$. 
We prove an explicit spectral lower bound that is sharp for $t=1$ in all spheres and for $t=2$ in every odd-dimensional sphere, yielding the first exact optimality results for spherical $t$-design curves with $t>1$. 
For even-dimensional spheres, we construct $2$-design curves whose lengths asymptotically match the lower bound as $d\to\infty$, and in $S^2$, we use numerical optimization and the calculus of variations to derive a candidate for the shortest $2$-design curve.
\end{abstract}

\section{Introduction}

Given a function $f$ on a measure space $(X,\mu)$, it might be intractable to explicitly compute the integral $\int_X f(x) \, d\mu(x)$.
As an alternative, one might seek a user-friendly measure $\nu$ on $X$ that approximates $\mu$ in the sense that
\begin{equation}\label{eqn:genquad}
     \int_X f(x) \, d\nu(x) \approx \int_X f(x) \, d\mu(x)
\end{equation}
for every function $f$ in some representative subclass $\cF$ of $\mu$-integrable functions.
Inspired by the geometric beauty of \emph{spherical $t$-designs}~\cite{McLaren:63,GoethalsS:79,Bannai:09} and applications in areas as diverse as mobile sampling~\cite{BenedettoW:00,GrochenigRUV:15,JamingNR:21,JayeM:22,UnnikrishnanV:13,UnnikrishnanV:13b,BoyerCCKW:16} and quantum information~\cite{XuWVA:25,PilatowskyCameoMCH:24,MarkSESCREC:24}, we build on the work in~\cite{EhlerG:23,EhlerGK:25,Ehler:25,EhlerG:25,Lindblad:24,Lindblad:25} that studies \emph{spherical $t$-design curves}.
Namely, our goal is to address~\eqref{eqn:genquad} in the case that $X$ is the unit sphere $S^d$ in $\bR^{d+1}$, $\mu$ is the uniform probability measure on $S^d$, $\cF$ is the space of restrictions to $S^d$ of the polynomials over $\bR^{d+1}$ of degree at most $t$, and $\nu$ is the pushforward of the normalized arclength probability measure of some closed rectifiable curve $\gamma\colon [0,2 \pi)\rar S^d$.
We focus on spherical $t$-design curves $\gamma$ of minimal or asymptotically minimal length.
Our main results give explicit lower bounds on the length of such curves as well as explicit and implicit constructions of optimal and near-optimal curves for small $t$.
In fact, we provide the first proofs of shortest spherical $t$-design curves with $t>1$.

\subsection{Previous work}
\subsubsection{Numerical quadrature}
As a simple example of our motivating equation~\eqref{eqn:genquad}, consider the integral form of the mean value theorem: if $f\colon [0,1] \rar \bR$ is continuous, there exists $c \in (0,1)$ such that 
\[
f(c) = \int_0^1 f(x)\, dx.
\]
While this result may be useful for undergraduate analysis homework, it does not help compute integrals since $c$ depends on $f$ and is non-constructive.
Instead, considering the general case of a probability measure space $(X,\mu)$, one might search for points $\{x_n\}_{n=1}^N$ in $X$ and positive weights $\{w_n\}_{n=1}^N$ that sum to $1$ such that 
\begin{equation}\label{eqn:quad}
    \sum_{n=1}^N w_n f(x_n) \approx \int_X f(x) \, d\mu(x)
\end{equation}
for all functions $f$ in some subclass $\mathcal{F}$ of $\mu$-integrable functions.
This corresponds to taking $\nu$ in~\eqref{eqn:genquad} to be an atomic probability measure.
Such \textit{numerical quadrature} (also known as \textit{cubature}) formulas have a long history in numerical integration~\cite{Radon:48}, in some sense going back to the very foundations of calculus~\cite{Newton:36}. 
(See also~\cite{deLaHarpeP:05} and the references therein.)
There is a universal approximation theorem for~\eqref{eqn:quad}:
for any path-connected topological space $X$, any Borel probability measure $\mu$ of full support, any finite-dimensional vector space $\cF$ of continuous $\mu$-integrable functions, and any sufficiently large $N$, there exists a choice $\{x_n\}_{n=1}^N$ of points in $X$ that satisfy~\eqref{eqn:quad} with equality for uniform weights $w_n = 1/N$; see~\cite{SeymourZ:84}. 
As with the mean value theorem, the existence of such a sampling set $\{x_n\}_{n=1}^N$ is non-constructive.

In the special case where $X$ is a compact homogeneous space, we may take $\mu$ to be the uniform probability measure over $X$, in which case it is perhaps not surprising that the sampling set must be somewhat uniformly spread over $X$.
In general, one may say that a uniformly spread sampling set allows for quadrature.
One such measure of the uniformity is \emph{discrepancy}; Koksma--Hlawka--type inequalities bound the approximation error in~\eqref{eqn:quad} as a product of the discrepancy of $\{x_n\}_{n=1}^N$ and the variation of $f$ (see, e.g.,~\cite{KuipersN:74,GrabnerKT:97}).
When the approximation in~\eqref{eqn:quad} is measured by upper and lower bounds with respect to $p$th powers, i.e., the existence of $0 < A \leq B<\infty$ such that for some $p \geq 1$,
\begin{equation}
\label{eqn:MZquad}
  A \int_X \abs{f(x)}^p \, d\mu(x) \leq \sum_{n=1}^N w_n \abs{f(x_n)}^p \leq B \int_X \abs{f(x)}^p \, d\mu(x),
\end{equation}
the inequalities are called \emph{Marcinkiewicz--Zygmund} as they generalize relationships between moments of independent random variables.
(See, e.g.,~\cite{Lubinsky:98} and the references therein.)
As physical limitations on measurement devices may constrain the sampling set $\{x_n\}_{n=1}^N$, it is possible to choose the weights $\{w_n\}_{n=1}^N$ so that~\eqref{eqn:MZquad} holds provided the sampling set is not too irregularly distributed~\cite{MhaskarNW:01}.

\subsubsection{Spherical $t$-designs}

Consider the case where $X$ is the unit sphere $S^d$ in $\mathbb{R}^{d+1}$, $\mu$ is the uniform probability measure on $S^d$, and $\cF$ is the vector space  $\Pi_d^t$ of polynomial functions over $S^d$ of degree at most $t$.
In this setting, if~\eqref{eqn:quad} holds with equality for uniform weights $w_n = 1/N$, then the sampling set $\{x_n\}_{n=1}^N$ is known as a \emph{spherical $t$-design}~\cite{McLaren:63,GoethalsS:79}.
Motivated by numerical integration, McLaren~\cite{McLaren:63} first introduced this concept in 1963.
In 1977, Delsarte, Goethals, and Seidel~\cite{DelsarteGS:77,GoethalsS:79} independently re-introduced the concept and coined the term ``spherical $t$-designs,'' viewing them as interesting objects in the context of algebraic combinatorics.
This mixed background highlights the simultaneous utility and beauty of spherical $t$-designs.
Interestingly, the (polynomial) discrepancy of a spherical $t$-design is inversely related to $t^2$, meaning that for large enough $t$, a spherical $t$-design satisfies an \textit{approximate} quadrature formula \`{a} la~\eqref{eqn:quad} for a much larger class of functions than the polynomials for which it satisfies an \textit{exact} quadrature formula.
(See~\cite{GrabnerKT:97} and the references therein.)
Some applications of spherical $t$-designs include hyperinterpolation~\cite{Sloan:95}, spherical partial differential equations~\cite{Wu:23}, spherical integral equations~\cite{AtkinsonB:87}, and spherical data processing~\cite{ChenW:18}.
For review articles on spherical $t$-designs, see the classic reference~\cite{Bannai:09} and the recent survey~\cite{AnZ:26}.

As a consequence of the uniform approximation theorem, for each $(t,d)$, there exist $t$-designs in $S^d$ of all sufficiently large sizes \cite{SeymourZ:84}, but it is interesting both in theory and practice to determine the smallest possible size.
In their first paper on the subject~\cite{DelsarteGS:77}, Delsarte, Goethals, and Seidel used Gegenbauer polynomials to prove that the number of points in a $t$-design on $S^d$ is bounded below by
\begin{equation}\label{eqn:dgsbd}
N_{t,d}
:=
\left\{
\begin{array}{cl}
\displaystyle
\binom{d+m}{d}+\binom{d+m-1}{d}
&\text{if } t=2m\\[8pt]
\displaystyle
2\binom{d+m}{d}
&\text{if } t=2m+1,
\end{array}\right.
\end{equation}
and they refer to a spherical $t$-design that saturates that bound as \emph{tight}.

There are a few infinite classes of tight spherical $t$-designs.
In $S^1$, the vertices of regular $(t+1)$-gons are precisely the tight $t$-designs.
For arbitrary $d \in \bN$, the tight $2$-designs in $S^d$ are the $d+2$ vertices of regular simplices, and the tight $3$-designs are the $2(d+1)$ vertices of cross polytopes.
Beyond these examples, tight designs seem to be rare.
In fact, for $d \geq 2$, there are no tight $t$-designs in $S^d$ when $t=6$ or $t \geq 8$, with the exception of a unique sporadic design with $(t,d) =(11,23)$~\cite{BannaiD:79,BannaiD:80,BannaiS:81}.
Since a tight $4$-design in $S^d$ exists if and only if a tight $5$-design in $S^{d+1}$ exists~\cite{DelsarteGS:77}, this leaves $t \in \{5,7\}$ as the interesting cases; however, there are infinitely many $d \geq 2$ for which there is no tight $5$- or $7$-design in $S^d$~\cite{BannaiMV:05,NebeV:13}.
Overall, it seems that for most choices of $(t,d)$, there does not exist a tight $t$-design in $S^d$.  

One may then ask: For a fixed $(t,d)$, what is the minimal cardinality of a spherical $t$-design in $S^d$?  
There is less literature concerning such results.  
Two papers prove the minimality and uniqueness of an $11$-design on $S^3$; this design has $120$ points, while the Delsarte--Goethals--Seidel bound~\eqref{eqn:dgsbd} yields $112$ \cite{Andreev:00,BoyvalenkovD:01}.
See also~\cite{SloaneHC:03,HardinS:92,Graef:online,Womersley:online,HardinS:96} for numerically approximated putatively minimal spherical $t$-designs.
A larger body of work is concerned with the asymptotically minimal size of designs where $d$ is fixed and $t \rar \infty$.
A tighter bound for larger $t$ is proven in~\cite{Yudin:97}.
Namely, let $N(d,t)$ denote the minimum cardinality of a spherical $t$-design in $S^d$.  
For a fixed $d \in \bN$,
\begin{equation}\label{eqn:asympptbd}
N(d,t) \geq \frac{2^{d}}{(d+1)^{d+1}}t^{d}\left( 1 + O(t^{-2/3})\right)
\end{equation}
as $t \rar \infty$.
Noting that~\eqref{eqn:asympptbd} (and also~\eqref{eqn:dgsbd}) gives a lower bound of the form $c_d t^d$, where $c_d$ is a constant depending only on $d$, one may ask if there is always a $t$-design in $S^d$ with $\asymp t^d$ points for fixed $d$ and large $t$.
In groundbreaking work~\cite{BondarenkoRV:13}, Bondarenko, Radchenko, and Viazovska showed that for every $N \geq c_d t^d$ with $c_d$ only dependent on $d$ there is a $t$-design in $S^d$ of size $N$, solving a conjecture of~\cite{KorevaarM:93}.
This proof is non-constructive, using Brouwer degree theory.

\subsubsection{Approximation by curves}

There are many applications in which curves (or higher-dimensional sets) are preferred over points. 
Applications to robotics motivated the authors of~\cite{RamamoorthyRW:08} to generalize the discrepancy measure from point sets to curves, which was then used to find paths for robots to traverse and sense two-dimensional regions.
Inspired in part by magnetic resonance imaging (MRI), mobile sampling involves recovering a bandlimited function on $\bR^d$ from its values on a curve (or more generally, a submanifold) with small path density~\cite{BenedettoW:00,GrochenigRUV:15,JamingNR:21,JayeM:22,UnnikrishnanV:13,UnnikrishnanV:13b,BoyerCCKW:16}.
In~\cite{BoyerCCKW:16}, the sampling curves are further restricted to those feasibly implementable in MRI machines; in particular, sampling along ``nice enough'' curves is much cheaper in an MRI than at a discrete set of points.
Mobile sampling corresponds to solving~\eqref{eqn:genquad} in the setting where $\mu$ is Lebesgue measure, $\nu$ is the $s$-dimensional Hausdorff measure restricted to an $s$-dimensional set, $\cF$ is Paley--Wiener space, and approximation is measured with respect to $L^2(\mu)$.

There are many more applications of approximation by curves.
For example, to help characterize realizable configurations of polymers and DNA molecules, various authors have sought to optimize the geometry of non-self-intersecting curves with bounded curvature in $\bR^{d+1}$ and $S^d$~\cite{CantarellaKS:02,GerlachM:11,GerlachM:11b,YuSC:21}.
The approximation of a 3D shape by a curve can also be used to improve 3D-printing methods~\cite{ChenSGCZ:17}.
Space-filling curves have even found applications in optimization~\cite{Goertzel:99}, compression~\cite{ChenYW:22}, and machine learning~\cite{TsinganosCCJS:21}. 
Principal curves generalize one-dimensional principal component analysis (PCA) by fitting curves to data sets to extract properties and interpolate data~\cite{HastieS:89,Hauberg:15,KeglKLZ:00,LeeKO:21}; this may be thought of as approximating the uniform probability measure of the point set with the pushforward of Lebesgue measure with respect to the fitted curve; the error is measured as the sum of squares of distances from the points to the fitted curve.

There are several ways to quantify the error in~\eqref{eqn:genquad} in cases where $\mu$ is a more general Borel measure and $\nu$ is supported on a curve.
For example, one may use discrepancy~\cite{EhlerGNS:21}, the $p$-Wasserstein distance for various choices of $p$~\cite{LebratGKW:19,ChauffertCKW:17,ChambolleDM:25,LuS:16}, or Sinkhorn divergences, which interpolate between maximum mean discrepancies and Wasserstein distances~\cite{FeydySVATP:19}.

\subsubsection{Spherical $t$-design curves}
Spherical $t$-design curves were introduced by Ehler and Gr\"ochenig in~\cite{EhlerG:23} to provide a theoretical groundwork for sampling and approximation by curves.
Recall that $\Pi_d^t$ denotes the vector space of polynomial functions over $S^d$ of degree at most $t$.

\begin{defn}
Let $\gamma\colon [0,2 \pi)\rar S^d$ be a Lipschitz parameterization of a closed, rectifiable curve, and let
\[
L(\gamma) 
:= \int_0^{2\pi} \norm{\gamma'(\theta)} \,d\theta
\]
denote the arclength of $\gamma$.
We say $\gamma$ is a \emph{spherical $t$-design curve} if for all $f \in \Pi_d^t$,
    \[
    \int_{S^d} f(x)\, d\sigma(x) = \frac{1}{L(\gamma)}\int_0^{2\pi} f(\gamma(\theta)) \, \norm{\gamma'(\theta)} \,d\theta,
    \]
where $\sigma$ denotes the uniform probability measure over $S^d$.
\end{defn}

Of particular interest are spherical $t$-design curves that (at least asymptotically in $t$ for fixed $d$) have minimal length.  
The following result~\cite{EhlerG:23} generalizes the asymptotic bounds for the cardinality of $t$-designs (e.g., \eqref{eqn:asympptbd}).

\begin{thm}\label{thm:asymplength}
    If $\gamma$ is a spherical $t$-design curve on $S^d$ with arclength $L(\gamma)$, then
    \[
    L(\gamma) \geq C_d t^{d-1}
    \]
    for some $C_d>0$ which is independent of $t$ and $\gamma$.
\end{thm}

The lower bound in Theorem~\ref{thm:asymplength} is asymptotically sharp for $d \leq 3$: 
\cite{EhlerG:23} takes spherical $t$-designs and connects the points with arcs of length $O(1/t)$ to form $t$-design curves.
Applying this procedure to the non-constructive existence results of~\cite{KorevaarM:93,BondarenkoRV:13} yields the existence of spherical $t$-design curves in $S^2$ which are asymptotically of minimal length as $t \rar \infty$ and $t$-design curves in $S^d$ for all $(t,d)$.
To get shorter curves, spherical design points were connected by geodesics in~\cite{EhlerGK:25}.
Unfortunately, this process does not generally result in a spherical $t$-design curve; it does, however, yield curves of length $\asymp t^{d-1}$ that satisfy Marcinkiewicz--Zygmund inequalities (\`a la~\eqref{eqn:MZquad}).
The construction also yields a useful initialization for numerical schemes.
In a couple of cases, the authors were able to exactify the curves output by the numerical algorithm. 
Instead of modifying $t$-design sets to make curves, Lindblad constructs $t$-design curves in $S^3$ by using the Hopf map to lift $\lfloor t/2\rfloor$-design curves in $S^2$~\cite{Lindblad:24}.
Applying this result to~\cite{EhlerG:23} gives $t$-design curves of asymptotically minimal length in $S^3$.
For an explicit construction of $t$-design curves (i.e., one that does not rely on a non-constructive proof of $t$-design sets), \cite{Ehler:25}~takes an edge--transitive polytope with $t$-homogeneous symmetry group and projects an Euler circuit of edges onto the surface of the sphere.

The Marcinkiewicz--Zygmund bounds in~\cite{EhlerGK:25} depend only on~$d$ and not~$t$. 
In~\cite{Lindblad:25}, a related concept called \emph{$\epsilon_t$-approximate $t$-design curves} was investigated, where as $t \rar \infty$, the curves provide more exact quadrature with error $\asymp 1/t$.
In this case, a substantial distinction between odd and even dimensions emerges. 
In odd dimensions, there again exist families of curves of length $\asymp t^{d-1}$, whereas in even dimensions, the construction obtained there has length $\asymp t^{2d-3}$.

In quantum error correction, spherical $t$-design curves may give rise to so-called tiger codes~\cite{XuWVA:25}.
The definition of spherical $t$-design curves is also related to Hilbert-space $k$-ergodicity~\cite{PilatowskyCameoMCH:24,MarkSESCREC:24}.
There are various generalizations of spherical $t$-design curves: weighted spherical $t$-design curves~\cite{EhlerG:23}, approximate spherical $t$-design curves~\cite{EhlerG:25,Lindblad:25}, and variants with measures of higher-~\cite{EhlerG:25} or mixed-dimensional~\cite{Ehler:25} support.

\subsection{Overview and roadmap}

In the original paper on spherical $t$-design curves~\cite{EhlerG:23}, the authors pose a few questions including the following two concerning the asymptotics of spherical $t$-design curves (with respect to $t \rar \infty$): ``What is the minimal order of the arc length of a $t$-design curve?'' and ``Are there $t$-design curves on $S^d$ achieving the optimal order of arc length?''
We seek to answer nonasymptotic versions of those questions, namely

\begin{prob}[main problem]\
\label{prob.main}
\begin{itemize}
\item[(a)]
What is the infimal arclength of $t$-design curves in $S^d$?
\item[(b)]
Which $t$-design curves in $S^d$ minimize arclength?
\end{itemize}
\end{prob}

The authors of~\cite{EhlerG:23} comment that the shortest $1$-design curves are precisely great circles, but no tight bounds have appeared in the literature for $t \geq 2$. In Theorem~\ref{thm.lower bound}, we provide a partial answer to Problem~\ref{prob.main}(a) with an explicit lower bound on the length of $t$-design curves.
This strengthening of Theorem~\ref{thm:asymplength} is a continuous analogue of the Delsarte--Goethals--Seidel bound~\eqref{eqn:dgsbd} for $t$-designs, and it is tight when $t=1$ and when $t=2$ and $d$ is odd (Example~\ref{ex.opt for t=2 and d odd}).
As with the Delsarte--Goethals--Seidel bound, we assume that this bound is not tight in general.
For example, we give conjectures for the shortest $2$-design curve in $S^2$ (Conjecture~\ref{conj:2S2}, Figure~\ref{fig.shortest 2-design curve in S^2}) and the shortest $3$-design curves in $S^d$ for $d$ odd (Conjecture~\ref{conj:3Sodd}), all of which fail to achieve equality in our bound.

\begin{figure}
\centering
\includegraphics[width=0.4\textwidth, trim=4cm 2cm 3.5cm 1.5cm, clip]{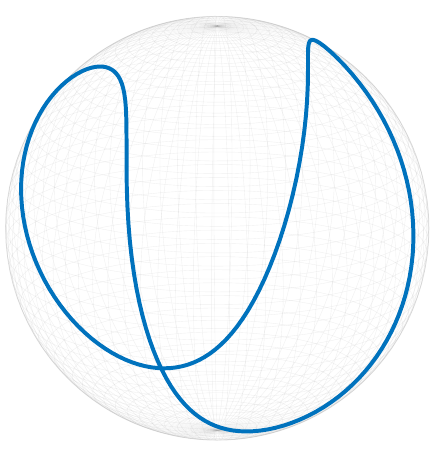}
\caption{The putatively shortest $2$-design curve in $S^2$; see Subsection~\ref{subsec.putatively shortest}}
\label{fig.shortest 2-design curve in S^2}
\end{figure}

The proof of Theorem~\ref{thm.lower bound} suggests one method to partially answer Problem~\ref{prob.main}(b), since any $t$-design curve in $S^d$ whose length saturates the lower bound must be minimal.
When our bound is not tight, a different approach is required.
We provide an ansatz (Lemma~\ref{lem.ansatz}) for $2$-design curves in $S^{2k}$ that incorporates intuition from the proof of Theorem~\ref{thm.lower bound} and exhibits structure we observed in numerically optimized piecewise-geodesic $2$-design curves in $S^{2k}$.
In Subsection~\ref{subsec.putatively shortest}, we leverage that ansatz to derive the putatively shortest $2$-design curve in $S^2$ (Conjecture~\ref{conj:2S2}) using the calculus of variations; see Figure~\ref{fig.shortest 2-design curve in S^2} for an illustration.
Our conjecturally minimal $3$-design curves in $S^d$ for $d$ odd (Conjecture~\ref{conj:3Sodd}) already appear in the $t$-design curve literature~\cite{EhlerGK:25,Lindblad:24} as well as the polytope literature~\cite{BarvinokN:08,Novik:23,Smilansky:85,Smilansky:90}; our approach in this case was to numerically optimize piecewise-geodesic curves, and our experiments consistently returned approximations to these putatively  shortest curves.

As another approach to Problem~\ref{prob.main}(b), in Theorem~\ref{thm.asymptotically shortest}, we use our general lower bound (Theorem~\ref{thm.lower bound}) and our ansatz (Lemma~\ref{lem.ansatz}) to prove asymptotically matching upper and lower bounds on the length of the shortest $2$-design curves in $S^{2k}$.
The upper and lower bounds are the same asymptotically in $d$; this differs from the usual asymptotic bounds that fix $d$ and send $t\to\infty$.
Our upper bound follows from an explicit construction (Theorem~\ref{thm.sine ansatz}) of short $2$-design curves in $S^{2k}$.

In Section~\ref{sec:genlowbd}, we prove the general lower bound on the lengths of $t$-design curves in $S^d$ and give examples of curves that saturate the bound.
Our ansatz for $2$-design curves in $S^{2k}$ and its application to construct such curves for all $k$ appear in Section~\ref{sec:ans}.
Section~\ref{sec:short2} contains upper and lower bounds on minimal length of $2$-design curves in $S^{2k}$ as well as a derivation of the putatively shortest $2$-design curve in $S^2$.
We conclude with a discussion (Section~\ref{sec:disc}) of key problems for future work.

\section{A general lower bound}\label{sec:genlowbd}

We begin with the main estimate of the paper: a lower bound on the length of an arbitrary spherical $t$-design curve in $S^d$.
As we will see, this lower bound is sharp in multiple regimes of $(t,d)$. 

\begin{thm}
\label{thm.lower bound}
Every $t$-design curve in $S^d$ has length at least
\[
2\pi\cdot\sqrt{\frac{A_{t,d}}{B_{t,d}}},
\]
where $A_{t,d}$ denotes the sum of the first $N_{t,d}-1$ terms in the sequence
\[
1,1,4,4,9,9,\ldots,n^2,n^2,\ldots,
\]
with $N_{t,d}$ given by~\eqref{eqn:dgsbd}, and
\[
B_{t,d}
:=
\left\{
\begin{array}{cl}
\displaystyle
\frac{m(d+m)}{d+2}
\bigg(
\binom{d+m}{d}+\binom{d+m-1}{d}
\bigg)
&\text{if } t=2m\\[12pt]
\displaystyle
\binom{d+m}{d}
\bigg(
1+\frac{2m(d+m+1)}{d+2}
\bigg)
&\text{if } t=2m+1.
\end{array}\right.
\]
\end{thm}

The proof is guided by the following idea.
(Here, we focus on the case where $t$ is even; the odd case is similar in spirit, but slightly more technical.)
Given a $t$-design curve, select a constant-speed parameterization
$\gamma\colon [0,2\pi)\to S^d$.
Identify a large family of mutually orthogonal low-degree polynomial functions $\{P_\alpha\}_\alpha$ over $S^d$, and then pull back these polynomials to form functions $\{P_\alpha\circ\gamma\}_\alpha$ over the circle.
Our bound follows from comparing two ways of measuring the total Dirichlet energy $\sum_\alpha \|(P_\alpha\circ\gamma)'\|_\circ^2$.
Here and throughout, $\operatorname{Leb}$ denotes Lebesgue measure and $\|\cdot\|_\circ$ denotes the norm in $L^2([0,2\pi),\frac{1}{2\pi}\operatorname{Leb})$.

What follows is a more detailed sketch.
First, since $\gamma$ is a $t$-design curve, the pullbacks are also mutually orthogonal, thereby forcing their Fourier support to occupy sufficiently many nonzero frequencies. 
Parseval's identity, integration by parts, and a water-filling estimate then deliver a lower bound on total energy in terms of the first few eigenvalues of the circle's Laplacian: 
\begin{equation}
\label{eq.main bound}
\sum_\alpha \|(P_\alpha\circ\gamma)'\|_\circ^2
\geq A_{t,d}.
\end{equation}
Alternatively, the total energy can be computed \textit{exactly} using the chain rule and the addition formula from the theory of spherical harmonics: 
\begin{equation}
\label{eq.main identity}
\sum_\alpha \|(P_\alpha\circ\gamma)'\|_\circ^2
=B_{t,d}\cdot\bigg(\frac{L(\gamma)}{2\pi}\bigg)^2.
\end{equation}
The result then follows from combining~\eqref{eq.main bound} with~\eqref{eq.main identity} and rearranging.

Recall that Delsarte, Goethals, and Seidel proved in~\cite{DelsarteGS:77} that every $t$-design in $S^d$ has cardinality at least $N_{t,d}$.
It perhaps comes as no surprise that we use the same spaces of spherical polynomials to prove our result. 
For an interpretable bound, Theorem~\ref{thm.lower bound} implies that every $t$-design curve in $S^d$ has length at least $c_d t^{d-1}$, where
\[
c_d
:=\frac{\pi}{2^{d-2}d!}\sqrt{\frac{d+2}{3}}.
\]
A lower bound with this scaling in $t$ previously appeared in Theorem~2.2 of~\cite{EhlerG:23} (repeated here as Theorem~\ref{thm:asymplength}), but not with an explicit constant.
It is conjectured that this lower bound is asymptotically sharp up to constants for each $d$, and to date, this is known to be true for $d\leq 3$; see~\cite{EhlerG:23,Lindblad:24}.
We note that there is an explicit lower bound on Marcinkiewicz--Zygmund curves of the form $C_d t^{d-1}$ reported in~\cite{EhlerG:25}, but when the upper and lower Marcinkiewicz--Zygmund bounds are equal, the resulting curve yields exact quadrature for various polynomials of degree $pt$, not just $t$, meaning the $C_d$ may be too large for a bound on $t$-design curves.

\begin{proof}[Proof of Theorem~\ref{thm.lower bound}]
Let $\gamma\colon[0,2\pi)\to S^d$ be a $t$-design curve with constant-speed parameterization.
As we will see, such a curve introduces interactions between the spectral theories of the circle and the sphere.
Let $\{\lambda_n\}_{n=1}^\infty$
denote the eigenvalues $0,1,1,4,4,9,9,\ldots$ of the circle's Laplacian
$-\frac{d^2}{d\theta^2}$, listed in monotonically increasing order with multiplicity.
Corresponding eigenfunctions are given by
\[
1,
~\cos\theta,
~\sin\theta,
~\cos2\theta,
~\sin2\theta,
~\ldots.
\]
Next, let $\{\lambda^{(k)}:=k(k+d-1)\}_{k=0}^\infty$ denote the eigenvalues of the spherical Laplacian $-\Delta_{S^d}$, listed \textit{without} multiplicity.
Here, $\lambda^{(k)}$ corresponds to the eigenspace $\mathcal{H}_k$, which consists of the spherical harmonics of degree $k$ in $S^d$.
For each $k$, we fix an $L^2(S^d,\sigma)$-orthonormal basis $\{Y^{(k)}_a\}_a$ of real polynomial functions for $\mathcal H_k$.

We will treat the parities of $t$ in cases, but first, we record two useful facts.

\medskip
\noindent
\textbf{Fact~1.}
Given a Bessel system $\{F_\alpha\}_\alpha$ in $L^2([0,2\pi),\frac{1}{2\pi}\operatorname{Leb})$ of Lipschitz functions with Bessel bound $1$ such that $\sum_\alpha \|F_\alpha\|_\circ^2=N$ for some positive integer $N$, it holds that
\[
\sum_\alpha \|F_\alpha'\|_\circ^2
\geq
\sum_{n=1}^N\lambda_n.
\]

Indeed, normalize eigenfunctions of the circle's Laplacian to form an orthonormal basis
$\{e_n\}_{n=1}^\infty$ for
$L^2([0,2\pi),\frac{1}{2\pi}\operatorname{Leb})$.
Parseval's identity and integration by parts together give
\[
\|F_\alpha'\|_\circ^2
=\sum_{n=1}^\infty |\langle F_\alpha',e_n\rangle_\circ|^2
=\sum_{n=1}^\infty |\langle F_\alpha,e_n'\rangle_\circ|^2
=\sum_{n=1}^\infty \lambda_n|\langle F_\alpha,e_n\rangle_\circ|^2.
\]
Summing both sides over $\alpha$ then gives
\[
\sum_\alpha \|F_\alpha'\|_\circ^2
=
\sum_{n=1}^\infty \lambda_n
\sum_\alpha |\langle F_\alpha,e_n\rangle_\circ|^2
=:
\sum_{n=1}^\infty \lambda_nx_n.
\]
Of course, each $x_n$ is nonnegative, and the Bessel bound gives
$x_n\leq 1$.
Meanwhile,
\[
\sum_{n=1}^\infty x_n
=
\sum_\alpha\sum_{n=1}^\infty |\langle F_\alpha,e_n\rangle_\circ|^2
=
\sum_\alpha\|F_\alpha\|_\circ^2
=
N.
\]
Since $\lambda_n$ is monotonically increasing in $n$, it follows that
$\sum_{n=1}^\infty\lambda_nx_n$ is minimized by taking
$x_1=\cdots=x_N=1$ and $x_n=0$ for each $n>N$.
Thus
\[
\sum_\alpha \|F_\alpha'\|_\circ^2
=
\sum_{n=1}^\infty \lambda_nx_n
\geq\sum_{n=1}^N\lambda_n,
\]
as claimed.

\medskip
\noindent
\textbf{Fact~2.}
For any $v\in S^d$, the projections of the vectors
$\{\nabla Y^{(k)}_a(v)\}_a$ onto the subspace $v^\perp$ form a tight frame for
$v^\perp$ with frame bound
\[
\frac{\lambda^{(k)}\operatorname{dim}\mathcal H_k}{d}.
\]

To be explicit, denoting $P_v=I-vv^\top$, we wish to show that
\[
S_v
:=
\sum_{a=1}^{\operatorname{dim}\mathcal H_k}
\langle\cdot,P_v\nabla Y^{(k)}_a(v)\rangle
P_v\nabla Y^{(k)}_a(v)
=
\frac{\lambda^{(k)}\operatorname{dim}\mathcal H_k}{d}P_v.
\]
First, we use symmetry to show that $S_v$ is a multiple of $P_v$.
Given any $R\in O(d+1)$ such that $Rv=v$, it holds that
$\{Y^{(k)}_a\circ R^{-1}\}_a$ is also an orthonormal basis for $\mathcal H_k$,
from which it follows that
\[
S_v
=
\sum_{a=1}^{\operatorname{dim}\mathcal H_k}
\langle\cdot,P_v\nabla(Y^{(k)}_a\circ R^{-1})(v)\rangle
P_v\nabla(Y^{(k)}_a\circ R^{-1})(v).
\]
Next, since
\[
P_v\nabla(Y^{(k)}_a\circ R^{-1})(v)
=
P_vR\nabla Y^{(k)}_a(v)
=
RP_v\nabla Y^{(k)}_a(v),
\]
we have $S_v=RS_vR^{-1}$.
Since $R$ was arbitrary, we conclude that $S_v=c_vP_v$ for some real number
$c_v$.
To determine $c_v$, we take the trace:
\[
dc_v
=
\operatorname{tr}(c_vP_v)
=
\operatorname{tr}(S_v)
=
\sum_{a=1}^{\operatorname{dim}\mathcal H_k}
\|P_v\nabla Y^{(k)}_a(v)\|^2.
\]
The right-hand side is rotation invariant, and so there is some $c\in\mathbb R$
such that $c_v=c$ for every $v\in S^d$.
We continue by averaging over $S^d$:
\[
dc
=
\int_{S^d}
\sum_{a=1}^{\operatorname{dim}\mathcal H_k}
\|P_x\nabla Y^{(k)}_a(x)\|^2\,d\sigma(x)
=
\sum_{a=1}^{\operatorname{dim}\mathcal H_k}
\int_{S^d}
\|P_x\nabla Y^{(k)}_a(x)\|^2\,d\sigma(x).
\]
Since $P_x\nabla Y^{(k)}_a(x)$ is the intrinsic gradient of
$Y^{(k)}_a|_{S^d}$, integration by parts gives
\[
dc
=
\sum_{a=1}^{\operatorname{dim}\mathcal H_k}
\int_{S^d}
Y^{(k)}_a(-\Delta_{S^d}Y^{(k)}_a)\,d\sigma
=
\lambda^{(k)}\operatorname{dim}\mathcal H_k.
\]
Thus $c=\frac{\lambda^{(k)}\operatorname{dim}\mathcal H_k}{d}$,
as claimed.

\medskip
\noindent
\textbf{Case I:} $t$ is even.
Write $t=2m$.
We first observe that the functions 
\[
f^{(k)}_a
:=Y^{(k)}_a\circ\gamma
\]
for $k\leq m$ are orthonormal in
$L^2([0,2\pi),\frac{1}{2\pi}\operatorname{Leb})$.
Indeed, note that 
\begin{itemize}
\item[(i)]
$\gamma$ is a constant-speed parameterization of the curve with speed
$\frac{1}{2\pi}L(\gamma)$ so that derivatives and the chain rule hold almost everywhere,
\item[(ii)]
$\gamma$ is a $t$-design curve and $Y^{(k)}_aY^{(\ell)}_b$ is a polynomial of
degree $k+\ell\leq 2m=t$, and
\item[(iii)]
$\{Y^{(k)}_a\}_{k,a}$ is orthonormal in $L^2(S^d,\sigma)$.
\end{itemize}
These facts imply the following equalities, respectively:
\[
\langle f^{(k)}_a,f^{(\ell)}_b\rangle_\circ
=
\frac{1}{L(\gamma)}\int_\gamma Y^{(k)}_aY^{(\ell)}_b\,ds
=
\int_{S^d}Y^{(k)}_aY^{(\ell)}_b\,d\sigma
=
\delta_{k,\ell}\delta_{a,b}.
\]
Thus $\{f^{(k)}_a\}_{k,a}$ is an orthonormal system, hence a
Bessel system with Bessel bound~$1$, and its total squared norm is
\[
\sum_{k=0}^m\sum_{a=1}^{\operatorname{dim}\mathcal{H}_k}\|f^{(k)}_a\|_\circ^2
=\sum_{k=0}^m\operatorname{dim}\mathcal{H}_k
=
\binom{d+m}{d}+\binom{d+m-1}{d}
=N_{2m,d}.
\]
It follows that
\[
\begin{aligned}
A_{2m,d}
&\leq
\sum_{k=0}^m
\sum_{a=1}^{\operatorname{dim}\mathcal H_k}
\| (f^{(k)}_a)' \|_\circ^2&&\textit{Fact~1}\\
&=
\sum_{k=0}^m
\sum_{a=1}^{\operatorname{dim}\mathcal H_k}
\frac{1}{2\pi}\int_0^{2\pi}
\bigg|
\frac{d}{d\theta}Y^{(k)}_a(\gamma(\theta))
\bigg|^2\,d\theta&&\textit{definition of \(\|\cdot\|_\circ\)}\\
&=
\sum_{k=0}^m
\frac{1}{2\pi}\int_0^{2\pi}
\sum_{a=1}^{\operatorname{dim}\mathcal H_k}
|\nabla Y^{(k)}_a(\gamma(\theta))\cdot\gamma'(\theta)|^2
\,d\theta\quad&&\textit{chain rule}\\
&=
\sum_{k=0}^m
\frac{1}{2\pi}\int_0^{2\pi}
\frac{\lambda^{(k)}\operatorname{dim}\mathcal H_k}{d}
\|\gamma'(\theta)\|^2\,d\theta&&\textit{Fact~2}\\
&=
\bigg(\frac{L(\gamma)}{2\pi}\bigg)^2\cdot
\frac{1}{d}
\sum_{k=0}^m
\lambda^{(k)}\operatorname{dim}\mathcal H_k.&&\textit{$\|\gamma'(\theta)\|=\frac{L(\gamma)}{2\pi}$}
\end{aligned}
\]
Recalling that
\[
\lambda^{(k)}=k(k+d-1),
\qquad
\operatorname{dim}\mathcal H_k
=
\binom{d+k}{d}-\binom{d+k-2}{d},
\]
a standard induction gives
\[
\frac{1}{d}
\sum_{k=0}^m
\lambda^{(k)}\operatorname{dim}\mathcal H_k
=
\frac{m(d+m)}{d+2}
\bigg(
\binom{d+m}{d}+\binom{d+m-1}{d}
\bigg)
=
B_{2m,d}.
\]
Thus,
\[
A_{2m,d}
\leq
\bigg(\frac{L(\gamma)}{2\pi}\bigg)^2\cdot B_{2m,d},
\]
which rearranges to the desired bound.

\medskip
\noindent
\textbf{Case II:} $t$ is odd.
Write $t=2m+1$, and let $\mathcal{U}_m$ denote the space of restrictions to $S^d$ of homogeneous polynomials of degree $m$ on $\mathbb{R}^{d+1}$.
Select an $L^2(S^d,\sigma)$-orthonormal basis
$\{P_i\}_i$ for $\mathcal{U}_m$, and denote the coordinate functions on
$S^d\subseteq\mathbb R^{d+1}$ by $X_1,\ldots,X_{d+1}$.
We use these to define functions over $[0,2\pi)$ by pulling back through $\gamma$:
\[
u_i:=P_i\circ\gamma,
\qquad
w_{\alpha i}:=(X_\alpha P_i)\circ\gamma.
\]

We claim that $\{u_i\}_{i}
\cup
\{w_{\alpha i}\}_{\alpha,i}$ is a Bessel system with Bessel bound~$1$ and total squared norm
$N_{2m+1,d}$.
First, the functions $\{u_i\}_i$ are orthonormal.
Indeed, $P_iP_j$ has degree at most $2m\leq t$, so the $t$-design condition gives
\[
\langle u_i,u_j\rangle_\circ
=
\int_{S^d}P_iP_j\,d\sigma
=
\delta_{i,j}.
\]
Second, the functions $w_{\alpha i}$ are orthogonal to every $u_j$.
Indeed, $P_i$ and $P_j$ have the same degree modulo $2$, and so $X_\alpha P_iP_j$ is odd
under $x\mapsto -x$.
Since $X_\alpha P_iP_j$ has degree at most $2m+1=t$, the design condition gives
\[
\langle w_{\alpha i},u_j\rangle_\circ
=
\int_{S^d}X_\alpha P_iP_j\,d\sigma
=
0.
\]
Next, for every $h\in L^2([0,2\pi),\frac{1}{2\pi}\operatorname{Leb})$,
\[
\sum_{\alpha=1}^{d+1}\sum_{i=1}^{\operatorname{dim}\mathcal{U}_m}
|\langle h,w_{\alpha i}\rangle_\circ|^2
=
\sum_{\alpha=1}^{d+1}\sum_{i=1}^{\operatorname{dim}\mathcal{U}_m}
|\langle (X_\alpha\circ\gamma)h,u_i\rangle_\circ|^2
\leq
\sum_{\alpha=1}^{d+1}
\|(X_\alpha\circ\gamma)h\|_\circ^2
=
\|h\|_\circ^2,
\]
where the last step uses the fact that $\sum_{\alpha=1}^{d+1}X_\alpha^2=1$ on $S^d$.
Since the span of the $w_{\alpha i}$'s is orthogonal to the span of the
$u_i$'s, adding the orthonormal system $\{u_i\}_i$ preserves the Bessel bound
$1$.
Finally, $\sum_i\|u_i\|_\circ^2=\operatorname{dim}\mathcal{U}_m$,
and
\begin{multline*}
\sum_{\alpha=1}^{d+1}\sum_{i=1}^{\operatorname{dim}\mathcal{U}_m}\|w_{\alpha i}\|_\circ^2
=
\frac{1}{2\pi}\int_0^{2\pi}
\sum_{\alpha=1}^{d+1}\sum_{i=1}^{\operatorname{dim}\mathcal{U}_m}
X_\alpha(\gamma(\theta))^2P_i(\gamma(\theta))^2
\,d\theta\\
=
\frac{1}{2\pi}\int_0^{2\pi}
\sum_{i=1}^{\operatorname{dim}\mathcal{U}_m}P_i(\gamma(\theta))^2\,d\theta
=
\sum_{i=1}^{\operatorname{dim}\mathcal{U}_m}\|u_i\|_\circ^2
=
\operatorname{dim}\mathcal{U}_m.
\end{multline*}
Thus, the total squared norm is 
\[
\sum_{i=1}^{\operatorname{dim}\mathcal{U}_m}\|u_i\|_\circ^2+\sum_{\alpha=1}^{d+1}\sum_{i=1}^{\operatorname{dim}\mathcal{U}_m}\|w_{\alpha i}\|_\circ^2
=
2\operatorname{dim}\mathcal{U}_m
=2\binom{d+m}{d}
=N_{2m+1,d},
\]
where the second step follows from the decomposition
\begin{equation}
\label{eq.decomposition into harmonics}
\mathcal{U}_m
=
\bigoplus_{\substack{0\leq k\leq m\\ k\equiv m\bmod 2}}
\mathcal H_k.
\end{equation}
This will allow us to use Fact~1 in what follows.

Next, we compute the total Dirichlet energy
\[
\sum_{i=1}^{\operatorname{dim}\mathcal{U}_m}\|u_i'\|_\circ^2
+
\sum_{\alpha=1}^{d+1}\sum_{i=1}^{\operatorname{dim}\mathcal{U}_m}\|w_{\alpha i}'\|_\circ^2.
\]
First, the chain rule, the decomposition~\eqref{eq.decomposition into harmonics}, and Fact~2 together give
\[
\sum_{i=1}^{\operatorname{dim}\mathcal{U}_m} |u_i'(\theta)|^2
=
\sum_{i=1}^{\operatorname{dim}\mathcal{U}_m} |\nabla P_i(\gamma(\theta))\cdot \gamma'(\theta)|^2
=
C_{m,d}\|\gamma'(\theta)\|^2,
\]
where
\[
C_{m,d}
:=
\frac{1}{d}
\sum_{\substack{1\leq k\leq m\\ k\equiv m\bmod 2}}
\lambda^{(k)}\operatorname{dim}\mathcal H_k.
\]
Next, recall from the addition formula that
\begin{equation}
\label{eq.add}
\sum_{i=1}^{\operatorname{dim}\mathcal{U}_m} P_i(v)^2={\operatorname{dim}\mathcal{U}_m}
\qquad
\forall\,v\in S^d.
\end{equation}
Thus, suppressing dependence on $\theta$, we have
\[
\begin{aligned}
\sum_{\alpha=1}^{d+1}\sum_{i=1}^{\operatorname{dim}\mathcal{U}_m} |w_{\alpha i}'|^2
&=
\sum_{\alpha=1}^{d+1}\sum_{i=1}^{\operatorname{dim}\mathcal{U}_m}
|\nabla(X_\alpha P_i)(\gamma)\cdot \gamma'|^2
&&\textit{chain rule}\\
&=
\sum_{\alpha=1}^{d+1}\sum_{i=1}^{\operatorname{dim}\mathcal{U}_m}
\left|
\gamma'_\alpha P_i(\gamma)
+
\gamma_\alpha(\nabla P_i(\gamma)\cdot \gamma')
\right|^2\quad
&&\textit{product rule}\\
&=
\sum_{i=1}^{\operatorname{dim}\mathcal{U}_m}
\big(
\|\gamma'\|^2P_i(\gamma)^2
+
|\nabla P_i(\gamma)\cdot \gamma'|^2
\big)
&&\textit{\(\gamma\cdot\gamma'=0\), \(\|\gamma\|=1\)}\\
&=
(\operatorname{dim}\mathcal{U}_m+C_{m,d})\|\gamma'\|^2.
&&\textit{Eq.\ \eqref{eq.add}, Fact~2}
\end{aligned}
\]
Combining the identities above, we then have
\[
\begin{aligned}
A_{2m+1,d}
&\leq
\sum_{i=1}^{\operatorname{dim}\mathcal{U}_m}\|u_i'\|_\circ^2
+
\sum_{\alpha=1}^{d+1}\sum_{i=1}^{\operatorname{dim}\mathcal{U}_m}\|w_{\alpha i}'\|_\circ^2
&&\textit{Fact~1}\\
&=
\frac{1}{2\pi}\int_0^{2\pi}
\bigg(
\sum_{i=1}^{\operatorname{dim}\mathcal{U}_m} |u_i'(\theta)|^2
+
\sum_{\alpha=1}^{d+1}\sum_{i=1}^{\operatorname{dim}\mathcal{U}_m} |w_{\alpha i}'(\theta)|^2
\bigg)\,d\theta\quad
&&\textit{definition of \(\|\cdot\|_\circ\)}\\
&=
\frac{1}{2\pi}\int_0^{2\pi}
(\operatorname{dim}\mathcal{U}_m+2C_{m,d})\|\gamma'(\theta)\|^2\,d\theta&&\textit{identities above}\\
&=
\bigg(\frac{L(\gamma)}{2\pi}\bigg)^2\cdot
(\operatorname{dim}\mathcal{U}_m+2C_{m,d}).
&&\textit{\(\|\gamma'(\theta)\|=\frac{L(\gamma)}{2\pi}\)}
\end{aligned}
\]
Finally, a standard induction argument gives
\[
C_{m,d}
=
\frac{m(d+m+1)}{d+2}\binom{d+m}{d},
\]
and so
\[
\operatorname{dim}\mathcal{U}_m+2C_{m,d}
=
\binom{d+m}{d}
\bigg(
1+\frac{2m(d+m+1)}{d+2}
\bigg)
=
B_{2m+1,d}.
\]
Thus,
\[
A_{2m+1,d}
\leq
\bigg(\frac{L(\gamma)}{2\pi}\bigg)^2\cdot B_{2m+1,d},
\]
which rearranges to the desired bound.
\end{proof}

\begin{ex}
\label{ex.great circle}
In the case where $t=1$, Theorem~\ref{thm.lower bound} reports that every $1$-design curve has length at least $2\pi$, which in turn implies that every great circle is a shortest $1$-design curve.
This is to be expected, and in fact, \cite{EhlerGK:25} states ``it is easy to show that great circles are shortest among all $1$-design curves.''
At the risk of beating a dead horse, we include yet another proof of this bound in the appendix; this proof, originally due to Stefan Steinerberger, makes use of classical ideas from integral geometry.
\end{ex}

\begin{ex}
\label{ex.opt for t=2 and d odd}
By inspecting the proof of the even-$t$ case of Theorem~\ref{thm.lower bound}, a closed curve with constant-speed parameterization $\gamma\colon[0,2\pi)\to S^{d}$ achieves equality in the lower bound when the functions $\{Y^{(k)}_a\circ \gamma\}_{k,a}$ form an orthonormal basis of bottom eigenfunctions of the circle's Laplacian $-\frac{d^2}{d\theta^2}$.
This is particularly plausible when $t=2$, since then we only need $\{X_a\circ \gamma\}_{a=1}^{d+1}$ to equal the bottom nonconstant eigenfunctions, where $X_a$ denotes the $a$th coordinate function.
When $d=2k-1$, this corresponds to taking
\[
\gamma(\theta)
\,:=\,\frac{1}{\sqrt{k}}\big(
\cos\theta,
~\sin\theta,
~\cos2\theta,
~\sin2\theta,
~\cdots,
~\cos k\theta,
~\sin k\theta
\,\big).
\]
This was first identified as a $2$-design curve for $S^3$ in Section~3 of~\cite{Lindblad:24}, where it was constructed using the Hopf map to lift a $1$-design curve (i.e., the great circle in the $xy$-plane) in $S^2$.
By the above discussion, $1$-design curve which is lifted and the resulting $2$-design curve are both of minimal length.
Example~2.2 of~\cite{EhlerGK:25} contains the formula for the curve as a $2$-design curve in all odd-dimensional spheres.

This curve has another name in polytope theory: the \emph{trigonometric moment curve}.  
Sampling \emph{any distinct} $n$ vertices of the moment curve in $S^{2k-1}$ and taking the convex hull results in a (not necessarily regular) polytope which is \emph{$k$-neighborly}, in some sense maximizing the maximum number of faces a polytope with $n$ vertices in $\bR^{2k}$ can have~\cite{Caratheodory:11,Gale:63}.

If we instead sample the curve at \emph{any} uniformly spaced (relative to $\bR/2\pi \bZ$) set of $2k+1$ points (rather than just the single set tested in Example~2.2 of~\cite{EhlerGK:25}), we will obtain a regular $2k$-simplex.
That is, this spherical $2$-design curve of minimal length is a smooth sweeping of tight spherical $2$-designs.
\end{ex}

\section{Ansatz for $t=2$, $d$ even}\label{sec:ans}

The previous section presented a general lower bound on the length of $t$-design curves in $S^d$, and then showed that equality is achieved when $t=1$ or when $t=2$ and $d$ is odd.
The remainder of the paper focuses on the case where $t=2$ and $d$ is even. 
What follows is an ansatz for $2$-design curves in even-dimensional spheres that we suspect is satisfied by the shortest of such curves.

\begin{lemma}
\label{lem.ansatz}
Given a positive integer $k$, consider any $\frac{\pi}{k+1}$-periodic $x_1,\ldots,x_k\in C^1(\mathbb{R})$ and any $\frac{\pi}{k+1}$-antiperiodic\footnote{A function $f\colon \bR \rar \bR$ is said to be \emph{$p$-antiperiodic} if $f(x+p)=-f(x)$ for all $x \in \bR$.} $z\in C^1(\mathbb{R})$ such that
\begin{equation}
\label{eq.norm condition}
\sum_{j=1}^kx_j(\theta)^2+z(\theta)^2
=1
\end{equation}
for every $\theta\in\mathbb{R}$ and
\begin{equation}
\label{eq.2 design condition}
\int_0^{2\pi}\bigg(x_j(\theta)^2-\frac{2}{2k+1}\bigg)\sqrt{\sum_{\ell=1}^k \big(x_\ell'(\theta)^2+\ell^2x_\ell(\theta)^2\big)+z'(\theta)^2}\,d\theta
=0
\end{equation}
for each $j\in\{1,\ldots,k\}$.
Then $\gamma\colon[0,2\pi)\to S^{2k}$ defined by
\begin{align*}
\gamma(\theta)
\,&:=\,\big(
\,x_1(\theta)\cos\theta,
~x_1(\theta)\sin\theta,
~x_2(\theta)\cos2\theta,
~x_2(\theta)\sin2\theta,
~\cdots\\
&\hspace{1.6in}
~\cdots,
~x_k(\theta)\cos k\theta,
~x_k(\theta)\sin k\theta,
~z(\theta)
\,\big)
\end{align*}
is a $2$-design curve.
\end{lemma}

\begin{proof}
First, \eqref{eq.norm condition} ensures that $\gamma$ is well defined in the sense that $S^{2k}$ is its codomain.
Next, since each $x_j$ and $z$ are $C^1$ and $2\pi$-periodic, it follows that the image of $\gamma$ is a rectifiable closed curve.
As we will see, $\gamma$ is a $2$-design thanks to the periodicity of each~$x_j$, the antiperiodicity of~$z$, and the specification~\eqref{eq.2 design condition}.

Observe that
\[
\|\gamma'(\theta)\|
=\sqrt{\sum_{j=1}^k \big(x_j'(\theta)^2+j^2x_j(\theta)^2\big)+z'(\theta)^2}\,.
\]
Since each $x_j$ and $z$ are $C^1$ and $\frac{\pi}{k+1}$-(anti)periodic, it holds that $\theta\mapsto\|\gamma'(\theta)\|$ is continuous and $\frac{\pi}{k+1}$-periodic, and so its Fourier coefficients are supported on $2(k+1)\mathbb{Z}$.
Meanwhile, for each $j\in\{1,\ldots,k\}$, the maps $\theta\mapsto x_j(\theta)\cos j\theta$ and $\theta\mapsto x_j(\theta)\sin j\theta$ have Fourier series supported on $\pm j+2(k+1)\mathbb{Z}$, which is disjoint from $2(k+1)\mathbb{Z}$, and so
\[
\int_0^{2\pi}x_j(\theta)\cos j\theta\,\|\gamma'(\theta)\|\,d\theta=0,
\qquad
\int_0^{2\pi}x_j(\theta)\sin j\theta\,\|\gamma'(\theta)\|\,d\theta=0.
\]
Similarly, since $z$ is $\frac{\pi}{k+1}$-antiperiodic, its Fourier series is supported on $(k+1)+2(k+1)\mathbb{Z}$, and so
\[
\int_0^{2\pi}z(\theta)\,\|\gamma'(\theta)\|\,d\theta=0.
\]
At this point, we have established that $\gamma$ is a $1$-design curve.

It remains to show that
\begin{equation}
\label{eq.frame operator}
\int_0^{2\pi}\bigg(\gamma(\theta)\gamma(\theta)^\top-\frac{1}{2k+1}I\bigg)\,\|\gamma'(\theta)\|\,d\theta
=0.
\end{equation}
The off-diagonal entries of \eqref{eq.frame operator} can be treated by the same analysis as above: when $i\neq j$, the map $\theta\mapsto\gamma_i(\theta)\gamma_j(\theta)$ has a Fourier series whose support is disjoint from $2(k+1)\mathbb{Z}$.
For the diagonal entries of \eqref{eq.frame operator}, we first observe that for each $j\in\{1,\ldots,k\}$, the maps $\theta\mapsto x_j(\theta)^2\cos^2j\theta-\frac{1}{2k+1}$ and $\theta\mapsto x_j(\theta)^2\sin^2j\theta-\frac{1}{2k+1}$ have Fourier series supported on $\{-2j,0,2j\}+2(k+1)\mathbb{Z}$.
This intersects $2(k+1)\mathbb{Z}$ only at the coset $0+2(k+1)\mathbb{Z}$, where our two maps have identical Fourier coefficients.
It follows that
\[
\int_0^{2\pi}\bigg(x_j(\theta)^2\cos^2j\theta-\frac{1}{2k+1}\bigg)\,\|\gamma'(\theta)\|\,d\theta
=\int_0^{2\pi}\bigg(x_j(\theta)^2\sin^2j\theta-\frac{1}{2k+1}\bigg)\,\|\gamma'(\theta)\|\,d\theta.
\]
In particular, both sides of the above identity equal their average, i.e., half of \eqref{eq.2 design condition}.
This treats all but the last diagonal entry of \eqref{eq.frame operator}.
For this final entry, we apply \eqref{eq.norm condition} to get
\begin{align*}
\int_0^{2\pi}\bigg(z(\theta)^2-\frac{1}{2k+1}\bigg)\,\|\gamma'(\theta)\|\,d\theta
&=\int_0^{2\pi}\bigg(1-\sum_{j=1}^kx_j(\theta)^2-\frac{1}{2k+1}\bigg)\,\|\gamma'(\theta)\|\,d\theta\\
&=-\sum_{j=1}^k\int_0^{2\pi}\bigg(x_j(\theta)^2-\frac{2}{2k+1}\bigg)\,\|\gamma'(\theta)\|\,d\theta,
\end{align*}
which vanishes by \eqref{eq.2 design condition}.
\end{proof}

The following remark identifies an important precursor to the above ansatz.

\begin{rem}
The previous literature contains $2$-design curves in $S^2$ that resemble the seam of a tennis ball.
Given any tennis ball curve $\gamma$ in $S^2$, let $P$ denote the $2$-dimensional subspace held invariant by the $D_4$ symmetry of $\gamma$.
Change coordinates so that $P$ is the $xy$-plane, and suppose the projection of $\gamma$ onto the $xy$-plane avoids the origin and that its composition with the argument map $(r\cos\theta,r\sin\theta)\mapsto\theta$ is injective.
This injectivity allows us to reparameterize $\gamma$ by $\theta$.
Let $x_1(\theta)$ denote the distance from the origin after projection, and let $z(\theta)$ denote the $z$-coordinate of $\gamma(\theta)$.
Then $\gamma(\theta)=(x_1(\theta)\cos\theta,x_1(\theta)\sin\theta,z(\theta))$, where $x_1$ and $z$ satisfy the hypotheses of Lemma~\ref{lem.ansatz}, except possibly~\eqref{eq.2 design condition}, which is satisfied if and only if $\gamma$ is a $2$-design curve.
In this way, the $2$-design curve in $S^2$ illustrated in Figure~1 of~\cite{EhlerG:23} can be reparameterized to satisfy the ansatz of Lemma~\ref{lem.ansatz}.
\end{rem}

What follows is another important instance of our general ansatz, which will help us evaluate our lower bound in the setting where $t=2$ and $d$ is even and large.

\begin{thm}
\label{thm.sine ansatz}
For every positive integer $k$, there exists $a_k\in(0,1)$ such that
\[
x_1(\theta)
=\cdots
=x_k(\theta)
=\sqrt{\frac{1-z(\theta)^2}{k}},
\qquad
z(\theta)=a_k\sin(k+1)\theta
\]
satisfy the hypotheses of Lemma~\ref{lem.ansatz}.
In particular, 
\[
a_1^2\in(0.820,0.821)
\qquad\text{and}\qquad
a_k^2\in\big(\tfrac{1}{k}-\tfrac{1}{2k^3},\tfrac{1}{k}+\tfrac{1}{4k^3}\big)
\quad
\forall\,k\geq2.
\]
\end{thm}

\begin{proof}
Note that $x_1,\ldots,x_k$, and $z$ satisfy all of the hypotheses of Lemma~\ref{lem.ansatz} except possibly~\eqref{eq.2 design condition}.
For this final constraint, we put $q:=a^2$ and change variables $u:=(k+1)\theta$ so that it suffices for the following quantity to vanish:
\[
H_k(q)
:=
\int_0^{2\pi}
\big(
\tfrac{1-q\sin^2u}{k}
-\tfrac{2}{2k+1}
\big)
\sqrt{
\tfrac{(k+1)(2k+1)}{6}\big(1-q\sin^2u\big)
+\tfrac{q(k+1)^2\cos^2u}{1-q\sin^2u}
}\,du.
\]
When $k=1$, one may use interval arithmetic to obtain
\[
H_1(0.820)>0
\qquad\text{and}\qquad
H_1(0.821)<0,
\]
and so the desired bounds on $a_1^2$ follow from the intermediate value theorem.

Now assume $k\geq 2$.
Again by the intermediate value theorem, it suffices to show
\[
H_k\bigg(\frac{1}{k}-\frac{1}{2k^3}\bigg)>0
\qquad\text{and}\qquad
H_k\bigg(\frac{1}{k}+\frac{1}{4k^3}\bigg)<0.
\]
To this end, we parameterize the intervening interval by
\[
q=\frac{1}{k}+\frac{c}{k^3},
\qquad
-\frac{1}{2}\leq c\leq \frac{1}{4}.
\]
We approximate the integrand of $H_k(\frac{1}{k}+\frac{c}{k^3})$ with a function that depends polynomially on $x:=1/k$.
The Taylor expansion of the integrand $F$ of $H_k(x+cx^3)$ about $x=0$ takes the form
\[
F(x,c,u)
=
C_1(u)\cdot x+C_2(u)\cdot x^2+C_3(c,u)\cdot x^3+R(x,c,u),
\]
and then integrating over $u$ gives
\begin{equation}
\label{eq.H expansion}
H_k\bigg(\frac1k+\frac{c}{k^3}\bigg)
=
-\frac{\sqrt3\pi(16c+1)}{48k^3}
+\underbrace{\int_0^{2\pi}
R(x,c,u)
\,du
}_{E(x,c)},
\qquad
x=\frac{1}{k}.
\end{equation}
One may use interval arithmetic to obtain
\[
\bigg|\frac{E(x,c)}{x^4}\bigg|
<1,
\qquad
0<x\leq\frac{1}{2},
\qquad
-\frac{1}{2}\leq c\leq\frac{1}{4}.
\]
Taking $c=-\frac{1}{2}$ and $c=\frac{1}{4}$ in~\eqref{eq.H expansion} and applying this bound, we get
\[
H_k\bigg(\frac{1}{k}-\frac{1}{2k^3}\bigg)
=\frac{7\sqrt{3}\pi}{48k^3}
+E\bigg(\frac{1}{k},-\frac{1}{2}\bigg)
>\frac{7\sqrt{3}\pi}{48k^3}
-\frac{1}{k^4}
>0,
\]
\[
H_k\bigg(\frac{1}{k}+\frac{1}{4k^3}\bigg)
=-\frac{5\sqrt{3}\pi}{48k^3}
+E\bigg(\frac{1}{k},\frac{1}{4}\bigg)
<-\frac{5\sqrt{3}\pi}{48k^3}+\frac{1}{k^4}
<0,
\]
where the inequalities use the fact that $k\geq 2$.
\end{proof}

\section{Short curves for $t=2$, $d$ even}\label{sec:short2}

In this section, we apply the general ansatz from Lemma~\ref{lem.ansatz} to construct short $2$-design curves in even-dimensional spheres.

\subsection{Asymptotically shortest curves for $t=2$, $d$ even}

In this section, we show that the lower bound in Theorem~\ref{thm.lower bound} is asymptotically sharp when $t=2$ and $d$ is even and large.
Indeed, this lower bound is nearly saturated by the family of curves constructed in Lemma~\ref{lem.ansatz} and Theorem~\ref{thm.sine ansatz}.

\begin{thm}
\label{thm.asymptotically shortest}
For every positive even integer $d$, each of the following holds:
\begin{itemize}
\item[(a)]
Every $2$-design curve in $S^d$ has length at least
\[
\frac{\pi}{\sqrt{3}}\cdot d
+\frac{5\pi}{2\sqrt{3}}
+1.44\cdot d^{-1}.
\]
\item[(b)]
There exists a $2$-design curve in $S^d$, constructed using
Lemma~\ref{lem.ansatz} and Theorem~\ref{thm.sine ansatz}, 
with length at most
\[
\frac{\pi}{\sqrt{3}}\cdot d
+\frac{5\pi}{2\sqrt{3}}
+5.86\cdot d^{-1}.
\]
\end{itemize}
\end{thm}

\begin{rem}
Impressively, the lower bound is on the order of $d$, and the difference between the upper and lower bounds is on the order of $d^{-1}$.
When $d$ is large, that means we get incredibly precise estimates for the length of the shortest $2$-design curve in $S^d$.
For example, when $d=1000$, this shortest length is $1818.337\pm0.003$.
\end{rem}

\begin{proof}[Proof of Theorem~\ref{thm.asymptotically shortest}]
For (a), Theorem~\ref{thm.lower bound} supplies the lower bound
\[
2\pi\cdot\sqrt{\frac{A_{2,d}}{B_{2,d}}}
=
\frac{\pi}{\sqrt{3}}\cdot
\sqrt{\frac{d(d+2)(d+1)+3(d+2)^2}{d+1}}
\geq
\frac{\pi}{\sqrt{3}}\cdot\bigg(
d+\frac52+\big(4\sqrt6-9\big)d^{-1}\bigg).
\]
The first two terms of the lower bound were selected to match the leading terms (in the large-$d$ regime) of the left-hand side, while the third term was selected so that the bound is sharp at $d=2$.
One can verify this final inequality holds for $d>2$ by squaring, multiplying by $\frac{4d^2(d+1)}{d-2}$, and then rearranging to a quadratic in $d$ with all positive coefficients.
Since $\frac{\pi}{\sqrt{3}}\big(4\sqrt6-9\big)>1.44$,
this proves (a).

For (b), write $d=2k$, and let $q_k:=a_k^2$ be the parameter supplied by Theorem~\ref{thm.sine ansatz}.  
For $q\in(0,1)$, denote the length of the corresponding curve by
\[
L_k(q)
:=
\int_0^{2\pi}
\sqrt{
\frac{(k+1)(2k+1)}{6}\big(1-q\sin^2u\big)
+\frac{q(k+1)^2\cos^2u}{1-q\sin^2u}
}\,du.
\]
For $k=1$, one may differentiate under the integral sign and then use interval arithmetic to show that $L_1'(q)>0$ for every $q\in[0,1]$.
Then
\[
L_1(q_1)
\leq L_1(0.821)
<\frac{2\pi}{\sqrt3}
+\frac{5\pi}{2\sqrt3}
+\frac{5.86}{2},
\]
where the strict inequality is also obtained by interval arithmetic.
Thus, the desired upper bound holds when $d=2$.

For $k\geq2$, Theorem~\ref{thm.sine ansatz} gives that there exists $c\in[-\frac{1}{2},\frac{1}{4}]$ such that $q_k=\frac1k+\frac{c}{k^3}$.
Imitating the proof of Theorem~\ref{thm.sine ansatz}, put $x:=1/k$ and denote the integrand of $L_k(x+cx^3)$ by $\frac{1}{x}\Psi(x,c,u)$.
The Taylor expansion of $\Psi(\cdot,c,u)$ about $x=0$ takes the form
\[
\Psi(x,c,u)
=C_0 + C_1(u)\cdot x + C_2(u)\cdot x^2+R(x,c,u),
\]
and then integrating $\frac{1}{x}\Psi(x,c,u)$ over $u$ gives
\[
L_k\bigg(\frac1k+\frac{c}{k^3}\bigg)
=
\frac{2\pi k}{\sqrt3}
+\frac{5\pi}{2\sqrt3}
+
\frac{17\pi}{16\sqrt{3}k}
+
\underbrace{\int_0^{2\pi}\frac{1}{x}
R(x,c,u)
\,du}_{E(x,c)},
\qquad x=\frac{1}{k}.
\]
One may use interval arithmetic to obtain
\[
\frac{E(x,c)}{x}<1,
\qquad
0< x\leq\frac12,
\qquad
-\frac12\leq c\leq\frac14.
\]
Since $k\geq2$, this bound on $E$ gives
\[
L_k(q_k)
\leq
\frac{2\pi k}{\sqrt3}
+\frac{5\pi}{2\sqrt3}
+
\bigg(
\frac{17\pi\sqrt3}{48}+1
\bigg)\frac1k.
\]
Since $d=2k$ and $\frac{17\pi\sqrt3}{24}+2<5.86$, this proves (b).
\end{proof}

\subsection{A putatively shortest curve for $t=2$, $d=2$}
\label{subsec.putatively shortest}

As a brief review, shortest spherical $2$-design curves are best understood in odd-dimensional spheres; see Example~\ref{ex.opt for t=2 and d odd}.
For even-dimensional spheres, Theorem~\ref{thm.asymptotically shortest} shows that the curves constructed in Lemma~\ref{lem.ansatz} and Theorem~\ref{thm.sine ansatz} are \textit{asymptotically} shortest $2$-design curves, but a nonasymptotic theory requires a different approach.
To help rectify this situation, we numerically hunted for a piecewise-geodesic $2$-design curve in $S^2$ of minimum length, and the results are summarized in Table~\ref{tab:S2-2-design-lengths}.

\begin{table}[h]
\centering
\begin{tabular}{ll}
\hline
curve & length \\
\hline
curve constructed by Lemma~\ref{lem.ansatz} and Theorem~\ref{thm.sine ansatz} 
    & $10.247113378452\ldots$ \\
Ehler--Gr\"ochenig curve $\gamma^{(2,a_2)}$; see~\cite{EhlerG:23}
    & $\phantom{1}9.786008023999\ldots$ \\
shortest known piecewise-geodesic $2$-design curve
    & $\phantom{1}9.743310534613\ldots$ \\
putatively shortest $2$-design curve
    & $\phantom{1}9.743299344305\ldots$ \\
\hline
\end{tabular}
\caption{Lengths of $2$-design curves in $S^2$.}
\label{tab:S2-2-design-lengths}
\end{table}

Notably, the shortest curve we found (which consisted of $1024$ geodesic pieces) is shorter than the $d=2$ version of the curve constructed in Lemma~\ref{lem.ansatz} and Theorem~\ref{thm.sine ansatz}, and even the ``Ehler--Gr\"ochenig curve'' from Proposition~3.1(ii) in~\cite{EhlerG:23}.
Much like the Ehler--Gr\"ochenig curve, this piecewise-geodesic curve resembles the seam of a tennis ball.
Accordingly, we expect the shortest $2$-design curve to take the form
\[
\gamma(\theta)
:=\Big(\,\sqrt{1-z(\theta)^2}\cos\theta,~\sqrt{1-z(\theta)^2}\sin\theta,~z(\theta)\,\Big),
\]
where $z$ is smooth and $\frac{\pi}{2}$-antiperiodic, and we may translate $z$ so that it is odd and increasing over $(0,\frac{\pi}{4})$, and furthermore, $z(\cdot+\frac{\pi}{4})$ is even.
(These specifications ensure that the resulting curve exhibits $D_4$ symmetry.)
By virtue of these symmetries, $z$ is determined by its restriction $w:=z|_{[0,\frac{\pi}{4}]}$.
In what follows, we use the calculus of variations to identify an arclength-critical $2$-design curve of this form.
The length of this putatively shortest $2$-design curve is given in the last row of Table~\ref{tab:S2-2-design-lengths}; see Figure~\ref{fig.z coord} for an illustration of the resulting $w$ and $z$.

\begin{figure}
\input{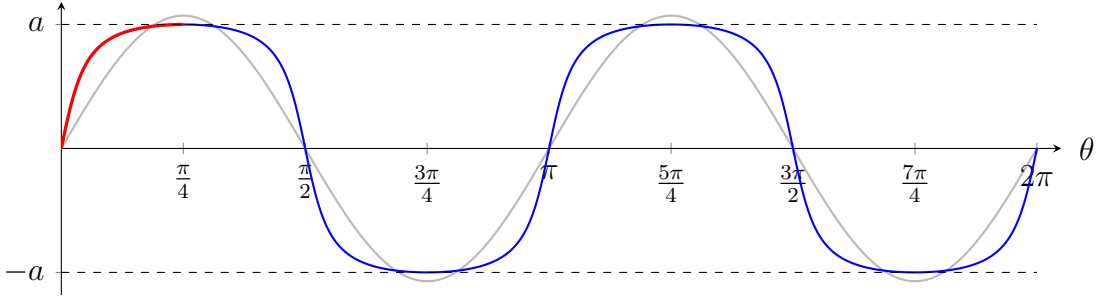}
\centering
\begin{tikzpicture}
\begin{axis}[
    width=0.95\textwidth,
    height=0.35\textwidth,
    axis lines=middle,
    xmin=0, xmax=2.05,
    ymin=-1, ymax=1,
    xlabel={$\theta$},
    xlabel style={
        at={(axis cs:2,0)},
        anchor=west,
        xshift=1em
    },
    xtick={0,0.25,0.5,0.75,1,1.25,1.5,1.75,2},
    xticklabels={
        $0$, $\frac{\pi}{4}$, $\frac{\pi}{2}$, $\frac{3\pi}{4}$,
        $\pi$, $\frac{5\pi}{4}$, $\frac{3\pi}{2}$,
        $\frac{7\pi}{4}$, $2\pi$
    },
    ytick={-0.845604052571,0,0.845604052571},
    yticklabels={$-a$,$0$,$a$},
]

\addplot[
    gray,
    thick,
    opacity=0.55,
    domain=0:2,
    samples=400
]
{0.905785627046754*sin(deg(2*pi*x))};

\addplot[blue, thick] coordinates {\zcoords};
\addplot[red, very thick] coordinates {\zhighlightcoords};

\addplot[black, dashed, domain=0:2, samples=2] {0.845604052571};
\addplot[black, dashed, domain=0:2, samples=2] {-0.845604052571};

\end{axis}
\end{tikzpicture}
\caption{Plot of the function $z$ that determines the putatively optimal $2$-design curve in $S^2$ via the ansatz in Lemma~\ref{lem.ansatz}. The portion in red is the restriction $w:=z|_{[0,\frac{\pi}{4}]}$. For comparison, the third coordinate $\theta\mapsto a_1\sin2\theta$ of the $2$-design curve constructed by Lemma~\ref{lem.ansatz} and Theorem~\ref{thm.sine ansatz} is also plotted in gray.}
\label{fig.z coord}
\end{figure}

For $\theta\in(0,\frac{\pi}{4})$, it holds that
\[
\|\gamma'(\theta)\|
=F(w(\theta),w'(\theta))
:=\sqrt{1-w(\theta)^2+\tfrac{w'(\theta)^2}{1-w(\theta)^2}}.
\]
We seek to minimize the length
\[
A[w]
:=8\int_0^{\pi/4}F(w(\theta),w'(\theta))d\theta
\]
subject to the constraint that $\gamma$ is a $2$-design curve:
\[
B[w]
:=\int_0^{\pi/4}\big(w(\theta)^2-\tfrac{1}{3}\big)F(w(\theta),w'(\theta))d\theta
=0.
\]
This suggests that we find a critical point of the Lagrangian
\[
L_\lambda[w]
:=A[w]+\lambda B[w]
=\int_0^{\pi/4}\underbrace{\Big(8+\lambda \big(w(\theta)^2-\tfrac{1}{3}\big)\Big)}_{q_\lambda(w(\theta))}F(w(\theta),w'(\theta))d\theta.
\]
The integrand of $L_\lambda[w]$ depends on $\theta$, but this dependence factors through $w$ and $w'$.
As such, we may apply the Beltrami identity, which in our case states that every critical point $w$ of $L_\lambda$ has the property that
\[
\frac{q_\lambda(w(\theta))(1-w(\theta)^2)}{F(w(\theta),w'(\theta))}
\]
is constant in $\theta$.
In fact, letting $a$ denote the maximum value of $w$, i.e., $w(\frac{\pi}{4})=a$, then since $w'(\frac{\pi}{4})=0$ by the smoothness of $z$, we know this constant equals
\[
c_{a,\lambda}
:=\frac{q_\lambda(a)(1-a^2)}{F(a,0)}
=\Big(8+\lambda\big(a^2-\tfrac{1}{3}\big)\Big)\sqrt{1-a^2}.
\]
This ordinary differential equation is nice enough that we can solve for the inverse function $\varphi_{a,\lambda}\colon w(\theta)\mapsto \theta$ by separation of variables:
\[
\varphi_{a,\lambda}(u)=\int_0^u\frac{c_{a,\lambda}}{(1-w^2)\sqrt{q_\lambda^2(w)(1-w^2)-c_{a,\lambda}^2}}\,dw.
\]
Thus, one may recover the function $w\colon[0,\frac{\pi}{4}]\to\mathbb{R}$ by first finding $a\in(0,1)$ and $\lambda\in\mathbb{R}$ that simultaneously satisfy 
\[
\varphi_{a,\lambda}(a)=\frac{\pi}{4}
\qquad\text{and}\qquad B[\varphi_{a,\lambda}^{-1}]=0
\]
and then put $w=\varphi_{a,\lambda}^{-1}$.
Numerically, this occurs at
\[
a=0.845604052571\ldots
\qquad\text{and}\qquad
\lambda=-10.223863780822\ldots.
\]
Furthermore, the Jacobian of $(a,\lambda)\mapsto(\varphi_{a,\lambda}(a),B[\varphi_{a,\lambda}^{-1}])$ is well conditioned at this point, and so in principle, one may use Newton--Kantorovich technology (\`{a} la~\cite{CohnKM:16}) to prove that an honest $2$-design curve exists nearby.

Recall that for the provably shortest $2$-design curves in $S^{2k-1}$, any evenly sampled set of $2k+1$ points yields a tight spherical $2$-design.
The situation for even $d$ is less satisfying but still yields some structure.
For any curve $\gamma$ that satisfies the general ansatz for $t=2$, $d=2k$ (Lemma~\ref{lem.ansatz}), one may use trigonometric identities and the (anti-)periodicity properties of $z$ and each $x_i$ to show that for any $\alpha \in [0,\frac{\pi}{k+1})$, sampling $\gamma$ at $\{\alpha + \frac{\ell \pi}{k+1} : \ell \in\{0, 1, \ldots, 2k+1\}\}$ yields a spherical $1$-design.  Numerically, for $\alpha \approx 0.050996 \pi$, this sampling of our putatively shortest $2$-design curve in $S^2$ is additionally a regular tetrahedron (i.e., a tight spherical $2$-design).

\section{Discussion}\label{sec:disc}

In this section, we identify several opportunities for future work.
Recall that the main problem of interest is Problem~\ref{prob.main}.
Perhaps surprisingly, we do not know the answer to the following fundamental question:

\begin{prob}
For every $t$ and $d$, does there exist a shortest $t$-design curve in $S^d$? 
\end{prob}

In this paper, we found curves of minimum length by achieving equality in a lower bound.
Examples~\ref{ex.great circle} and~\ref{ex.opt for t=2 and d odd} lead us to wonder:

\begin{prob}\label{prob:thm4tight}
For which pairs $(t,d)$ is the lower bound in Theorem~\ref{thm.lower bound}
sharp?
When the bound is sharp, can one classify the curves that achieve equality?
\end{prob}

Recall that for most choices of $(t,d)$, there does not exist a tight spherical $t$-design in $S^d$~\cite{BannaiD:79,BannaiD:80,BannaiS:81,BannaiMV:05,NebeV:13},  i.e., no such design has cardinality saturating~\eqref{eqn:dgsbd}.
Since being a spherical $t$-design curve requires more structure, one might expect even fewer $(t,d)$ for which Theorem~\ref{thm.lower bound} is sharp.
It might even be possible to completely classify the curves that achieve equality.

Next, we highlight a conjecture suggested by the previous section:

\begin{conj}\label{conj:2S2}
The curve derived in Subsection~\ref{subsec.putatively shortest} is the shortest $2$-design curve in~$S^2$, up to isometry and reparameterization.
\end{conj}

In pursuit of partial progress, we suggest the following:

\begin{prob}
Can the curve derived in Subsection~\ref{subsec.putatively shortest} be certified as
a strict local minimizer of arclength among $2$-design curves in $S^2$, up to isometry and reparameterization?
\end{prob}

Judging by the complexity of this curve, we suspect that shortest $2$-design curves in even-dimensional spheres are generally not easy to describe in closed form.
However, the success we had with Lemma~\ref{lem.ansatz} suggests the following:

\begin{prob}
Do shortest $2$-design curves in even-dimensional spheres always satisfy the ansatz in Lemma~\ref{lem.ansatz} (up to isometry and reparameterization)?
\end{prob}

While this paper mostly focused on $t\leq2$, we also performed extensive numerical experiments on the $3$-design curves described in Section~3 of \cite{Lindblad:24} and Example~2.2 of~\cite{EhlerGK:25} for small $d$, which compel us to pose the following:

\begin{conj}\label{conj:3Sodd}
The curve $\gamma\colon[0,2\pi)\to S^{2k-1}$ defined by
\[
\gamma(\theta)
\,:=\,\frac{1}{\sqrt{k}}\big(
\cos\theta,
~\sin\theta,
~\cos3\theta,
~\sin3\theta,
~\cdots,
~\cos (2k-1)\theta,
~\sin(2k-1)\theta
\,\big)
\]
is the shortest $3$-design curve in $S^{2k-1}$, up to isometry and reparameterization.
\end{conj}
In addition to numerical evidence, these curves share many properties with the provably optimal $2$-design curves in $S^{2k-1}$ (Example~\ref{ex.opt for t=2 and d odd}).
The curves in Conjecture~\ref{conj:3Sodd} are known as \emph{symmetric moment curves} in polytope theory; sampling them over a finite set and taking the convex hull leads to \emph{locally $k$-neighborly polytopes}~\cite{BarvinokN:08,Novik:23,Smilansky:85,Smilansky:90}.
By sampling our curve at \textit{any} set of $4k$ uniformly spaced points in $\bR/2\pi \bZ$ (rather than just the single set tested in Example~2.2 of~\cite{EhlerGK:25}), we will obtain a cross-polytope, i.e., a tight $3$-design in $S^{2k-1}$.
In so many words, this spherical $3$-design curve is a smooth sweeping of tight spherical $3$-designs.
Recall that the optimal $2$-design curves from Example~\ref{ex.opt for t=2 and d odd} are the trigonometric moment curves, where any sampling results in a $k$-neighborly polytope, and these curves are also smooth sweepings of tight spherical $2$-designs.
As a further connection to the optimal $2$-design curves,  both classes of curves are constructed in $S^3$ as liftings from the equatorial great circle in $S^2$~\cite{Lindblad:24}.
Perhaps the existing literature (e.g., \cite{BarvinokN:08,Vinzant:11,BarvinokLN:13,AdamsBF:20,Novik:23}) that characterizes the polytopes and orbitopes formed from these curves could help one approach this conjecture.

If one were to construct short $3$-design curves in even-dimensional spheres (perhaps following an ansatz akin to Lemma~\ref{lem.ansatz}), then one could answer the following:

\begin{prob}
How sharp is the lower bound in Theorem~\ref{thm.lower bound} when $t=3$ and $d$ is large?
\end{prob}

In general, we do not know how to tackle Problem~\ref{prob.main} in the absence of a saturating lower bound.
Perhaps one could use ideas from the proof~\cite{Andreev:00,BoyvalenkovD:01} of the minimal (non-tight) $11$-design on $S^3$ to study sporadic $t$-design curves.
Furthermore, it might be interesting to develop repositories of short piecewise-geodesic $t$-design curves as in the $t$-design case~\cite{SloaneHC:03,HardinS:92,Graef:online,Womersley:online}.
Also, the tighter lower bounds on the cardinality of $t$-designs in~\cite{Yudin:97} might be leveraged to improve our lower bound.
Given the results in~\cite{Lindblad:25} on approximate spherical designs, one might wonder, in the case where $d$ is even, if the minimum attainable arclength scales like $t^{2d-3}$  as $t \rar \infty$.
Note that $2d-3=d-1$ when $d=2$, which in turn is the only even $d$ for which we know asymptotically shortest $t$-design curves~\cite{EhlerG:23}. 

Based on the known shortest $t$-design curves, we are led to wonder:

\begin{prob}
Are shortest $t$-design curves necessarily smooth?
symmetric? contain a minimum-cardinality $t$-design set?
\end{prob}

Every explicit construction in this paper has the property that there is at least one sampling from the curve that is a tight $t$-design. 
When $(t,d) \in \{(1,d),(2,2k+1),(3,2k+1)\}$, the curves are smooth sweepings of tight $t$-designs.
Recall from the discussion at the end of Section~\ref{sec:short2} that for any curve $\gamma$ that satisfies the general ansatz for $t=2$ and $d=2k$ (Lemma~\ref{lem.ansatz}), sampling at the points $\{\alpha + \frac{\ell \pi}{k+1} : \ell \in\{0, 1, \ldots, 2k+1\}\}$ for any $\alpha \in [0,\frac{\pi}{k+1})$ yields a spherical $1$-design and that, at least numerically, one choice of $\alpha$ will yield a tight spherical $2$-design for the putatively shortest $2$-design curve in $S^2$ (Subsection~\ref{subsec.putatively shortest}).
In addition, if we pick 
\[
\alpha = \tfrac{1}{k+1} \arcsin\big( \tfrac{1}{a_k \sqrt{2k+1}}\big)
\]
to sample the asymptotically optimal curves of Theorem~\ref{thm.sine ansatz}, then we also obtain a regular $2k+1$-simplex, i.e., a tight spherical $2$-design in $S^{2k}$.

Interestingly, closed curves that contain a minimum-cardinality $t$-design set enjoy a different lower bound on arclength.
For example, a tight $3$-design in $S^2$ is the vertex set of an orthoplex, which consists of $6$ points of pairwise geodesic distances $\pi$ and $\frac{\pi}{2}$.
In general, any closed curve that contains $N$ points with minimum pairwise distance~$\delta$ has length at least~$N\delta$.
Thus, any closed curve on $S^2$ that contains a tight $3$-design set has length at least $6\cdot \frac{\pi}{2} = 3 \pi$.
The previous paragraph leads one to consider $3$-design curves that contain a tight $3$-design.
Plugging $(t,d)=(3,2)$ into Theorem~\ref{thm.lower bound} yields $2 \pi \sqrt{19/9}<3\pi$, and so no $3$-design curve that contains a tight $3$-design in $S^2$ can saturate the lower bound in Theorem~\ref{thm.lower bound}.

We end with a potpourri of questions.

\begin{prob}
    Is there a connection between longest ropes of fixed thickness in the sphere \cite{GerlachM:11,GerlachM:11b} and shortest $t$-design curves \`a la the connection between spherical codes and designs in~\cite{DelsarteGS:77}?
\end{prob}

\begin{prob}
Are spherical $t$-design curves optimal approximators of uniform measures with respect to other quantifications of approximation strength like $p$-Wasserstein metrics or Sinkhorn divergences?
\end{prob}

\begin{prob}
Consider an $s$-dimensional set $X$ in $S^d$ (satisfying some additional regularity conditions as in~\cite{EhlerG:25}) such that the normalized Hausdorff measure restricted to $X$ yields an exact quadrature of $\Pi_d^t$. Is there a sharp, nonasymptotic, parity-sensitive lower bound on the $s$-dimensional Hausdorff measure of $X$ that specializes to the Delsarte--Goethals--Seidel bound~\eqref{eqn:dgsbd}
when $s=0$ and to Theorem~\ref{thm.lower bound} when $s=1$?
\end{prob}

Care will need to be taken to either exclude positive-measure retracing when restricting to $s=1$ or allowing a multiplicity function.
Results on Marcinkiewicz--Zygmund sampling measures~\cite{EhlerG:25} already provide an asymptotic bound on the order of $t^{d-s}$.

\section*{Acknowledgments}

DGM was supported by NSF DMS 2220304.
DGM thanks OpenAI for complimentary access to ChatGPT-5.5~Pro.
Throughout this research project, the authors used ChatGPT to find relevant literature, to write and run exploratory code, to explore possible proof strategies, and to give feedback on writing. 
All mathematical claims, citations, and computations were checked by the authors, who take full responsibility for the content of this paper.

\appendix

\section{Shortest curves for $t=1$}\label{app:app}

\begin{prop}
Every $1$-design curve in $S^d$ has length at least $2\pi$.
\end{prop}

\begin{proof}
Fix a $1$-design curve $\gamma$.
For each $v\in S^d$, the linear functional $\langle \cdot,v\rangle$ is a polynomial of degree $1$, and so
\[
\frac{1}{L(\gamma)}\int_\gamma \langle \cdot,v\rangle ds=\int_{S^d}\langle \cdot,v\rangle d\sigma
=0.
\]
Then $\langle\cdot,v\rangle$ is neither strictly positive nor strictly negative on all of $\gamma$.
Since $\gamma$ is connected, the intermediate value theorem gives that $v^\perp$ intersects $\gamma$ nontrivially.

Now draw $v$ uniformly from $S^d$.
With probability $1$, the subspace $v^\perp$ intersects $\gamma$ finitely often with all transversal intersections.
Since $\gamma$ is closed, the number of intersections is even (and thus $\geq2$, since the intersection is nontrivial).
The spherical Crofton formula then gives
\[
L(\gamma)
=\pi\cdot\mathbb{E}\big[\#(\gamma\cap v^\perp)\big]
\geq2\pi.
\qedhere
\]
\end{proof}

\end{document}